\pdfoutput=1 

\documentclass[a4paper, 12pt]{article}

\usepackage[utf8]{inputenc}
\usepackage[T1]{fontenc}
\usepackage[a4paper, margin=2.5cm]{geometry}
\usepackage{needspace}

\usepackage{libertine} 
\usepackage{inconsolata} 

\usepackage{amsmath, amsthm, amssymb}
\usepackage{graphicx}
\usepackage{enumerate}
\usepackage{authblk}
\usepackage{thmtools}
\usepackage{thm-restate}
\usepackage[colorlinks=true, citecolor=red]{hyperref}
\usepackage{cleveref}
\usepackage{float}
\usepackage{tikz}
\usetikzlibrary{positioning, shapes, arrows, calc}

\renewenvironment{abstract}
{\small\vspace{-1em}
\begin{center}
\bfseries\abstractname\vspace{-.5em}\vspace{0pt}
\end{center}
\list{}{
\setlength{\leftmargin}{0.6in}%
\setlength{\rightmargin}{\leftmargin}}%
\item\relax}
{\endlist}

\declaretheorem[name=Theorem, numberwithin=section]{theorem}
\declaretheorem[name=Lemma, sibling=theorem]{lemma}
\declaretheorem[name=Proposition, sibling=theorem]{proposition}
\declaretheorem[name=Definition, sibling=theorem]{definition}

\declaretheorem[name=Claim, sibling=theorem]{claim}

\declaretheorem[name=Observation, style=remark, sibling=theorem]{observation}
\declaretheorem[name=Question, style=remark, sibling=theorem]{question}
\def\cqedsymbol{\ifmmode$\lrcorner$\else{\unskip\nobreak\hfil
\penalty50\hskip1em\null\nobreak\hfil$\lrcorner$
\parfillskip=0pt\finalhyphendemerits=0\endgraf}\fi} 
\newcommand{\cqed}{\renewcommand{\qed}{\cqedsymbol}}

\let\leq\leqslant
\let\geq\geqslant

\title{Bipartite Tur\'an number of paths and other trees\footnote{This work was done while the second author was doing a bachelor internship in LaBRI with the other two authors. We gratefully acknowledge support from both institutions that made this possible.}}

\author[1]{Marthe Bonamy}
\author[2]{Théotime Leclere}
\author[1]{Timothé Picavet}
\affil[1]{CNRS, LaBRI, Université de Bordeaux, Bordeaux, France.}
\affil[2]{Université Paris-Saclay, ENS Paris-Saclay, 91190 Gif-sur-Yvette, France}
\date{\today}

\begin{document}

\maketitle

\begin{abstract}
We solve a recent question of Caro, Patk\'os and Tuza by determining the exact maximum number of edges in a bipartite connected graph as a function of the longest path it contains as a subgraph and of the number of vertices in each side of the bipartition. This was previously known only in the case where both sides of the bipartition have equal size and the longest path has size at most $5$. We also discuss possible generalizations replacing ``path'' with some specific types of trees.

\end{abstract}


\section{Introduction}\label{sec:intro}

We study the maximum number of edges in a graph avoiding certain subgraphs. This follows the tradition of Tur\'an’s theorem~\cite{turan1941}, which determines the maximum number of edges in an $n$-vertex graph avoiding a complete subgraph of given order. This result initiated the systematic study of extremal problems with forbidden subgraphs and was later extended by the Erd\H{o}s--Stone theorem~\cite{ErdosStone1946}, which asymptotically gives the Tur\'an number of every non-bipartite forbidden graph, i.e. the maximum number of edges in a $n$-vertex graph avoiding a fixed subgraph. For bipartite graphs, however, the Erd\H{o}s--Stone theorem provides no information on the Tur\'an number.
The bipartite case remains open and has motivated several central conjectures, most notably the long-standing Erd\H{o}s--S\'os conjecture, which states that every $n$-vertex graph excluding a tree on $k$ vertices has less than $(k-2)n/2$ edges. 
This conjecture has spawned new research on the Tur\'an number for specific graph families.
\cite{gallai1959maximal} gives asymptotics for connected graphs excluding a path of fixed size, while
\cite{Gyarfas1984Paths} gives the exact Tur\'an number for bipartite graphs excluding a path of fixed size.
Recently, Caro, Patk\'os and Tuza~\cite{caro2025bipartite} focused on the case where the host graph is both bipartite and connected. We refer the reader to that paper and~\cite{yuan2017variation} for more extended history.

To study these questions in the bipartite connected setting, Caro, Patk\'os and Tuza introduce the following notation.
\begin{definition}[\cite{caro2025bipartite}]
    For any fixed integers $a,b$ and graph $H$, we let $ex_{b,c}(a,b,H)$ be the maximum number of edges in a connected bipartite graph with part sizes $a$ and $b$ and with no $H$ as subgraph.
\end{definition}

Caro, Patkós and Tuza~\cite{caro2025bipartite} studied the value of $ex_{b,c}$ for trees. They found the exact value of this number for all trees with $6$ vertices or less, for double-stars and for some case of spiders. Notably, we mention the following result, which will be useful to us later:

\begin{proposition}[Theorem 7 of~\cite{caro2025bipartite}]\label{pr:Caro}
    For every $n > 3$, we have $ex_{b,c}(n,n,P_5) = ex_{b,c}(n,n,P_6) = 2n-1$.
\end{proposition}

\paragraph{Contributions.}
In the conclusion of~\cite{caro2025bipartite}, the authors posed several open problems, two of which directly motivate our work.
\begin{question}[Question 16 of \cite{caro2025bipartite}]
    \label{qu:question16deCPT}
    Determine at least asymptotically $ex_{b,c}(n,n,P_k)$ for any fixed $k$.
\end{question}
\begin{question}[Part of Question 14 of \cite{caro2025bipartite}]
    \label{qu:question14deCPT}
    Determine $ex_{b,c}$ for specific families of trees.
\end{question}
We solve the first question completely, and provide a partial answer to the second. Namely, we provide the exact value of the connected bipartite Tur\'an number for every path. 

\begin{restatable}{theorem}{thmPaths}\label{th:main}
    For every $k\geq 3$ and every $b \geq a\geq k$, we have $ex_{b,c}(a,b,P_{2k})=ex_{b,c}(a,b,P_{2k-1}) = (k-2)(b-1) + a$.
\end{restatable}

This answers Question~\ref{qu:question16deCPT} in a strong way: both because we provide the exact upper bound instead of the asymptotic value, and because we do not need the restriction to connected bipartite graphs that are balanced. 

We address Question~\ref{qu:question14deCPT} for the case of brooms. The broom $S_{p,d*1}$ is composed of a path of length $p$ and a star of $d$ branches, which is a tree on $d+1$ leaves.
We extend the result of paths to brooms where the star is larger than the path.

\begin{restatable}{theorem}{thmBrooms}\label{th:d>n}
    For every $k,d \geq 2$ and for every $n$ such that $n\geq d^2/2$ and $d > 2k$, we have $ex_{b,c}(n,n,S_{2k,d}) = ex_{b,c}(n,n,S_{2k+1,d}) = nd $.
\end{restatable}

We also provide an upper bound when the star is at most half the size of the path.
\begin{restatable}{theorem}{thmBroomsTwo}\label{th:d<n}
    For every $k> d \geq 2$ and for every $n$, we have $ex_{b,c}(n,n,S_{2k+1,d*1}) \leq 2nk + 1$.
\end{restatable}

When discussing those results, we heard that He, Salia and Zhu independently made similar progress in a follow-up to~\cite{caro2025bipartite}. Their paper~\cite{salia} should be available on arXiv on the same day as this one.


\section{Lower bound of Theorem~\ref{th:main}}\label{sec:lowermain}

We first prove the lower bound, with Proposition~\ref{th:borneinf_general}.

\begin{proposition}\label{th:borneinf_general}
    For every $k\geq 3$ and every $b\geq a\geq k$, we have $ex_{b,c}(a,b,P_{2k})\geq ex_{b,c}(a,b,P_{2k-1}) \geq (k-2)(b-1) - a$.
\end{proposition}

\begin{proof}
   
    The first inequality is immediate: we focus on the second one and exhibit a construction.
    Let $a ,b,k \in \mathbb{N^*}$ such that $b \geq a \geq k \geq 3$. 
    We build a connected bipartite graph $G$ of part sizes a and b as follows (see also Figure~\ref{fig:lowerbound}). 
    
    Informally, we consider a complete bipartite graph $K_{k-2,b}$ and add $a-(k-2)$ leaves adjacent to one vertex in a way that matches the expected sizes of the parts. More formally, we let $A = \{u_1,\ldots,u_a\}$ and $B = \{v_1,\ldots,v_b\}$ be the two sides of the bipartition of $G$. For every $i,j \in [1,b]\times[1,k-2]$, we set $u_iv_j \in E(G)$ and for every $j \in [k-1,a]$, we set $u_nv_j \in E(G)$.
    
    Note that $|E(G)| = b(k-2) + a - (k-2) = (b-1)(k-2) + a$.    Furthermore, $G$ is $P_{2k-1}$-free. 
     Indeed, if we take $G' = G\backslash \{v_{k-1},\dots,v_n\}$, the graph $G'$ is the complete bipartite graph $K_{k-2,b}$ and any longest path of this graph has $2k - 3$ vertices. However, only $u_n$ has neighbours in $G$ that are not in $G'$, and they all have degree $1$ in $G$.\end{proof}

     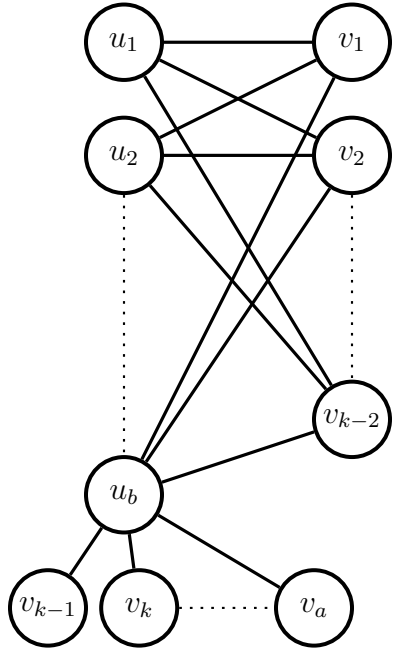
\begin{figure}
\begin{center}
        \begin{tikzpicture}
          \tikzset{sommet/.style={circle, draw, minimum size=1cm, line width=1.5pt, inner sep=0pt}}
          \node[sommet] (u_1) at (1,7.5) {$u_1$};
          \node[sommet] (u_2) at (1,6) {$u_2$};
          \node[sommet] (u_b) at (1,1.5) {$u_b$};
        
           \node[sommet] (v_1) at (4,7.5) {$v_1$};
          \node[sommet] (v_2) at (4,6) {$v_2$};
          \node[sommet] (v_k-2) at (4,2.5) {$v_{k-2}$};
        
           \node[sommet] (v_k-1) at (0,0) {$v_{k-1}$};
          \node[sommet] (v_k) at (1.2,0) {$v_k$};
          \node[sommet] (v_a) at (3.5,0) {$v_a$};
        
          \draw[thick,dash pattern=on 1 pt off 3pt] (u_2) -- (u_b);
          \draw[thick,dash pattern=on 1 pt off 3pt] (v_2) -- (v_k-2);
          \draw[thick,dash pattern=on 1 pt off 3pt] (v_k) -- (v_a);
          
          \draw[very thick] (u_1) -- (v_1);
          \draw[very thick] (u_1) -- (v_2);
          \draw[very thick] (u_2) -- (v_1);
          \draw[very thick] (u_2) -- (v_2);
          \draw[very thick] (u_1) -- (v_k-2);
          \draw[very thick] (u_2) -- (v_k-2);
          \draw[very thick] (u_b) -- (v_k-2);
          \draw[very thick] (u_b) -- (v_1);
          \draw[very thick] (u_b) -- (v_2);
        
          \draw[very thick] (u_b) -- (v_k-1);
          \draw[very thick] (u_b) -- (v_k);
          \draw[very thick] (u_b) -- (v_a);
          
    \end{tikzpicture}
    \end{center}
    \caption{The construction in Proposition~\ref{th:borneinf_general}.}
    \label{fig:lowerbound}
\end{figure}

Note that for $a=b=n$ we have $|E(G)| = (n-1)(k-2) + n = n(k-1) - (k-2)$, thus $ex_{b,c}(n,n,P_{2k})\geq ex_{b,c}(n,n,P_{2k-1}) \geq (k-1)n - (k-2)$.

\section{Upper bound of Theorem~\ref{th:main}}\label{sec:uppermain}

We start with the statement and the high-level arguments. The lower-level arguments are presented in Section~\ref{subsec:details}.

\begin{theorem}\label{th:bornesup_general}
    For every $k\geq 3$ and every $b\geq a\geq k$, we have $ex_{b,c}(a,b,P_{2k-1}) \leq ex_{b,c}(a,b,P_{2k}) \leq (k-2)(b-1) + a$.
\end{theorem}

Since $ex_{b,c}(a,b,P_{2k-1}) \leq ex_{b,c}(a,b,P_{2k})$ follows from the definition, we focus on the second part, i.e. $ex_{b,c}(a,b,P_{2k}) \leq (k-2)(b-1) + a$. We first need the following ingredients.

Lemma~\ref{lem:inductionUnbalanced} guarantees the existence of a vertex of reasonably small degree in the largest side of the bipartition, with the extra property that deleting this vertex does not disconnect the graph.

\begin{lemma}\label{lem:inductionUnbalanced}
    For every $k \geq 3$, for every connected bipartite $P_{2k}$-free graph $G$ with bipartition $(A,B)$ such that $|B| \geq |A| \geq k$, there exists a vertex $x \in B$ such that $d(x) \leq k-2$ and $G\setminus \{x\}$ is connected.
\end{lemma}

Lemma~\ref{lem:inductionBalanced} guarantees that any large enough connected balanced bipartite graph contains two ``irrelevant'' vertices. Intuitively, those two vertices can safely be deleted for induction purposes.

\begin{lemma}\label{lem:inductionBalanced}  
    For every $n \geq k \geq 3$, for every connected balanced bipartite graph $G$ of size $2n$ with no $P_{2k}$, there exist two vertices $x, y$ such that $d(x)+d(y) \leq k-1$ and $G \setminus \{x,y\}$ is connected and balanced.
\end{lemma}

We then argue that the theorem holds in the base case (i.e. $b=a=k$).

\begin{lemma}\label{lem:Casdebase}
    For every $k \geq 3$, for every connected bipartite $P_{2k}$-free graph $G$ with bipartition $(A,B)$ such that $|B|= |A| = k$, we have $|E(G)|\leq (k-1)^2 + 1$.
\end{lemma}

Using these, we are ready to prove Theorem~\ref{th:bornesup_general}.

\begin{proof}[Proof of Theorem~\ref{th:bornesup_general}]
    We proceed by induction on $k+a+b$. Let $G$ be a connected bipartite $P_{2k}$-free graph with bipartition $(A,B)$ and with $|A|=a$ and $|B|=b$.
    If $b>a$, then by Lemma~\ref{lem:inductionUnbalanced}, there exists a vertex $x \in B$ such that $d(x)\leq k-2$ and $G\setminus \{x\}$ is connected. The graph $G \setminus \{x\}$ is also bipartite and $P_{2k}$-free. By induction, it has at most $(k-2)(b-2) + a$ edges. Consequently, the graph $G$ itself has at most $(k-2)(b-2) + a+(k-2)=(k-2)(b-1) + a$, as desired.

    Therefore, we may assume from now on that $a=b$, i.e. the graph $G$ is balanced. If $a>k$, by Lemma~\ref{lem:inductionBalanced} there are two vertices $x,y$ such that $d(x)+d(y)\leq k-1$ and $G\setminus \{x,y\}$ is balanced. By induction, it has at most $(k-2)(b-2) + (a-1)$ edges. Consequently, the graph $G$ itself has at most $(k-2)(b-2) + (a-1)+k-1=(k-2)(b-1)+a$ edges, as desired.

    We are left with the case $a=b=k$, which is precisely the context of Lemma~\ref{lem:Casdebase}. We note that in that case $(k-2)(b-1)+a=(k-1)^2 - (k-1) + k=(k-1)^2+1$, hence the conclusion.
\end{proof}

\subsection{Proof details}\label{subsec:details}

Claim~\ref{cl:Bleaf} guarantees that a tree has a leaf in its largest side of the bipartition (under a reasonable extra assumption when both sides are the same size).

\begin{claim}\label{cl:Bleaf}
    For any tree $T$, let $A, B$ be the two sides of the bipartition of $T$ and let $r$ the root of $T$. If $r \in A$ and $|A| \leq |B|$ or $r\in B$ and $|A| < |B|$, then $T$ has a leaf $b$ in $B$.    
\end{claim}

\begin{proof}[Proof of Claim~\ref{cl:Bleaf}]
    We proceed by induction on height $h(T)$ of the tree: 
    \begin{itemize}
        \item If $h(T)$ = 0, then $T$ has a single node $r$ and $r\in B$. 

        \item If $h(T) \geq 1 $, let $r$ be the root of $T$.
        \begin{itemize}
            \item  If $r\in B$ and $|A| < |B|$, then there exists a connected component $T'$ of $T\setminus r$ such that $|A \cap V(T')| \leq |B \cap V(T')|$ and such that the root of $T'$ is in $A$. By induction, $T'$ has a leaf $b$ in $B$. The vertex $b$ is also a leaf for $T$.
            \item If $r\in A$ and $|A| \leq |B|$, then there exists a connected component $T'$ of $T\setminus r$ such that $|A \cap V(T')| < |B \cap V(T')|$ and the root of $T'$ is in $B$. By induction, $T'$ has a leaf $b$ in $B$. The vertex $b$ is also a leaf for $T$.
        \end{itemize}
    \end{itemize}
    
    \vspace{-0.7cm}\cqed 
\end{proof}

\begin{proof}[Proof of Lemma~\ref{lem:inductionUnbalanced}]
    Let $G$ be a connected bipartite graph of bipartition $(A,B)$ such that $|B| \geq |A| \geq k$ and $G$ is $P_{2k}$-free.
    Let $a\in A$ and let $T$ be a Depth-First-Search tree on $G$ with root $a$.
    By Claim~\ref{cl:Bleaf}, the tree $T$ has a leaf $b\in B$. Let $P = u_1u_2\ldots u_{\ell-1}u_\ell$ be the path from $a=u_1$ to $b=u_\ell$ in $T$. 

    The graph $G$ is $P_{2k}$-free and $P$ is a path on $\ell$ vertices, so $\ell \leq 2k-1$. In fact, since the extremities of $P$ lie in different sides of the bipartition, $\ell$ is even hence $\ell \leq 2k-2$. Note that $T\setminus \{b\}$ is connected, hence $G\setminus \{b\}$ is connected. Furthermore, given that $T$ is a DFS tree, we have $N(b) \subset V(P)$, which together with the fact that $G$ is bipartite implies $d(b) \leq k-1$. 

    If $d(b)\leq k-2$, the conclusion holds immediately by taking $x=b$. Therefore, we may assume $d(b) = k-1$. It follows that $\ell = 2k-2$ and $b$ is adjacent to $a$, hence forming a cycle $C$. In fact, $N(b)=V(P)\cap A$. Every vertex $v$ not in the cycle is adjacent to $C$, and all the neighbours of $v$ are also on the cycle, otherwise we could create a path on $2k$ vertices. 
    Let $b'\in B\setminus V(P)$ and let $a'\in A\setminus V(P)$. As argued above, we have $N(b') \subseteq V(P) \cap A$ and $N(a') \subseteq V(P) \cap B$.
    Assume for sake of contradiction that $d(b') \geq k-1$.
    Then $N(b') = V(P) \cap A = N(b)$ and $d(b')=k-1$. Let $i$ be such that $u_i \in N(a')$ (such an $i$ exists by connectivity and since $N(a') \subseteq V(P) \cap B$). 
    Then as $V(P)\cap A = N(b)$, the vertex $b$ is adjacent to $u_{i-1}$ and the path $a'u_i\ldots u_{i+1}\ldots u_\ell u_1 \ldots u_{i-1} b'$ contains $2k$ vertices, contradicting the $P_{2k}$-freeness.
    Therefore, we have $d(b') \leq k-2$.
    Furthermore, notice that $b'$ is not a cut-vertex, as all its neighours are on the path $P$. 
    So we can take $x = b'$.
\end{proof}

\begin{proof}[Proof of \Cref{lem:inductionBalanced}]
    Let $A,B$ be the two sides of the bipartition of $G$.
    Let $P=(u_1,u_2,\dots,u_p)$ be a longest path in $G$.
    Note that $N(u_1) \cup N(u_p) \subseteq V(P)$, which implies that $G\setminus \{u_1,u_p\}$ is connected. 
    In addition, since $G$ is bipartite, it follows that $d(u_1),d(u_p) \leq \frac{|P|}2$. 
    There are two cases, depending on whether $u_1$ and $u_p$ belong to the same set of $\{A,B\}$.

    \begin{claim}\label{cl:same_part}
        If $u_1 \in A$ and $u_p \in B$, we can take $x=u_1$ and $y=u_p$.
    \end{claim}
    \begin{proof}
        Note that $p=|P|$ is even in this case.
        If $u_1$ and $u_p$ do not satisfy the necessary conditions, then it must be that $d(u_1)+d(u_p)\geq k$.
        Then, by pigeonhole principle, as $N(u_1),N(u_p)\subseteq V(P)$ are disjoint and as their union has size at least $k$, there exists some $j$ such that $u_{2j}\in N(u_1)$ and $u_{2j-1}\in N(u_p)$. 
        This means that $C =u_1 u_2 \dots u_{2j-1} u_p u_{p-1} \dots u_{2j} u_1$ is a cycle with the same vertex set as $P$.
        Since $G$ is connected and has size more than $|P|$, there exists a vertex $v \in V(G)\setminus V(P)$ and a vertex $u_j \in V(P)$ such that $vu_j \in E(G)$. This contradicts the fact that $P$ has maximum length.
        
        \vspace{-0.4cm}\cqed 
    \end{proof}

    Now, let us assume that we are not under the conditions of \Cref{cl:same_part}, i.e. without loss of generality that $u_1, u_p \in A$. 
    We consider a Depth-First-Search tree $T$ on $G$ with the path $P$ already visited before the start of the DFS, with order $u_1, u_2, \dots, u_p$.
    This means that in particular, $P$ is a path of $T$.
    By \Cref{cl:Bleaf}, as $u_1\in A$ is the root of $T$ and $|A|=|B|$, then $T$ has a leaf $b\in B$.
    Consider the shortest path $P_b$ in $T$ from $b$ to any vertex of $T$.
    Let $1\leq j\leq k-1$, $\ell$ and the $v_i$'s be such that $P_b = b v_1 v_2\ldots v_{\ell-1} v_\ell$ with $v_\ell=u_j$.
    Note that $\ell\geq 1$, as $P$ contains no leaf in $B$.
    The following fact is well-known.
    \begin{observation}\label{obs:DFS}
        In a DFS tree of a graph, any neighbor of a leaf $x$ of the DFS tree is in the root-to-leaf path of $x$.
    \end{observation}
    It follows immediately from \Cref{obs:DFS} that $N(b) \subseteq \{u_1,\ldots,u_j\} \cup V(P_b)$.
    Furthermore, $u_1$ cannot be adjacent to $b$ since $N(u_1)\subseteq P$.

    \begin{claim}\label{cl:partialneighbors}
        The neighbors of $u_1$ along the path $u_j, \dots, u_p$ are in $\{u_j,u_{j+\ell+1}, \dots, u_{p-\ell}\}$. 
    \end{claim}
    \begin{proof}
        Consider the possible neighbors of $u_1$ along the path $u_j, \dots, u_p$.
        \begin{itemize}
            \item $u_1$ cannot have a neighbor $u_i$ with $j+1 \leqslant i \leqslant j+\ell$. Otherwise, the path $$b v_1 \dots v_{\ell-1} u_j \dots u_1 u_i\dots u_p$$ would have length $\ell + j+ p - i+1 \geq p+1$, giving a contradiction.
            \item $u_1$ cannot have a neighbor $u_i$ with $i \geq p-\ell+1$. Otherwise, the path $$b v_1 \dots v_{\ell-1} u_j \dots u_1 u_i \dots u_{j+1}$$ would have length $\ell + j + i - j \geqslant p+1$, giving a contradiction.
        \end{itemize}
        
        \vspace{-0.7cm}\cqed 
    \end{proof}

    Assume for sake of contradiction that $d(u_1)+d(b) \geq k$.
    We now prove the following claim.
    \begin{claim}\label{cl:neighbors}
        The number of vertices adjacent to $u_1$ or $b$ in $u_2,\dots,u_{j-1}$ is at least $j/2$.
    \end{claim}
    \begin{proof}
        Notice that $j+\ell+1$ is always odd, i.e. $u_{j+\ell+1}\in A$.
        Using Claim~\ref{cl:partialneighbors} and the fact that $u_1$ is only adjacent to vertices in $B$, we get that $u_1$ has at most $\lfloor  \frac{p-j-2 \ell}2 \rfloor  = \lfloor \frac{p-j}2 \rfloor - \ell$ neighbors in $u_{j+1}, \dots, u_p$. As $b\in B$, the same type of argument (without using Claim~\ref{cl:partialneighbors}) gives us that $b$ has at most $\lceil \frac{\ell-1}2 \rceil$ neighbors in $P_b\setminus \{v_\ell\}$.
        Thus the number of vertices adjacent to $u_1$ or $b$ in $u_2,\dots,u_j$ is at least
        \[
            k-\left\lfloor  \frac{p-j}2 \right\rfloor  + \ell - \left\lceil \frac{\ell-1}2 \right\rceil 
            \geq k-\left\lfloor  \frac{2k-1-j}2 \right\rfloor  + \ell - \left\lceil \frac{\ell-1}2 \right\rceil
            = \left\lceil \frac{1+j}2 \right\rceil + \left\lfloor \frac{\ell+1}2 \right\rfloor\geq \frac{j}{2}+1
        \]
        The statement of the claim follows from the last inequality.
        
        \vspace{-0.4cm}\cqed 
    \end{proof}
    

    A \emph{block} is a pair of vertices $(u_{2i-1},u_{2i})\in A\times B$ for some $2\leq i\leq j/2$.
    It is \emph{good} if $u_{2i-1} b\in E(G)$ and $u_{2i} u_1\in E(G)$.
    If there exists some good block, we can derive a contradiction on the maximality of $p$ by considering the path $v_{\ell-1}\ldots v_1 b u_i \ldots u_1 u_{i+1} \ldots u_p$.
    We now prove that $d(u_1)+d(b) \leq k - 1$ with the following.
    \begin{claim}
        There exists a good block.
    \end{claim}
    \begin{proof}
        Suppose for the sake of contradiction that there is no good block.
        Then every block can only be adjacent to one of $u_1$ and $b$.
        Moreover, every vertex in $\{u_3,\dots, u_{j-1}\}$ is contained in some block.
        Therefore, as there are $\left\lfloor\frac{j-3}{2}\right\rfloor$ blocks, the number of vertices adjacent to $u_1$ or $b$ in $\{u_2,\dots,u_{j-1}\}$ is at most $\left\lfloor\frac{j-3}{2}\right\rfloor + 1 = \left\lfloor\frac{j-1}{2}\right\rfloor < \frac{j}{2}$.
        This is a contradiction with \Cref{cl:neighbors}.
        
        \vspace{-0.4cm}\cqed 
    \end{proof}
    
    Thus $d(u_1)+d(b) \leq k - 1$. Furthermore, $G - \{u_1,b\}$ is connected. Indeed, $b$ is a leaf of the spanning tree $T$ and $u_1$ is the root but has only one child $u_2$. Thus $T - \{u_1,b\}$ is still a spanning tree of $G- \{u_1,b\}$.
    Hence we can take $x=u_1$ and $y=b$.
\end{proof}

\begin{proof}[Proof of Lemma~\ref{lem:Casdebase}]
    We proceed by induction on $k$.
    The case $k\leq 3$ has already been solved in \cite{caro2025bipartite}: see \Cref{pr:Caro}.
    Therefore, let $k > 3$, $G$ be a connected balanced bipartite graph of size $2k$ with no $P_{2k}$ and let $A, B$ be the two sides of the bipartition of $G$.
    We assume for a contradiction that $|E(G)| \geq (k-1)^2 + 2$.
    Let $P = u_1 \ldots u_p$ be a path of maximum length in $G$. Then $p< 2k$ and, without loss of generality, $u_1 \in A$.
    We first get some degree constraints on $u_1$.
    \begin{claim}\label{cl:max_degree}
        The vertex $u_1$ is not a cut-vertex of $G$ and $d(u_1)\leq k-2$.
    \end{claim}
    \begin{proof}
        The vertex $u_1$ is not a cut-vertex, otherwise we could extend $P$, contradicting its maximality.
        In addition, $N(u_1) \subseteq V(P)$, otherwise $P$ can be extended.
        Suppose for the sake of contradiction that $d(u_1) \geq k-1$.
        Then the length of $P$ is at least $2k-2$ and $V(P) \cap B \subseteq N(u_1)$. As $G$ contains $k$ vertices in $B$, there exists a vertex $b \in B$ not in the path.
        There are two cases.
        \begin{itemize}
            \item If $b$ is adjacent to a vertex $u_i \in A$ of the path, then we can construct  a new path $b u_i u_{i-1} \dots u_1 u_{i+1} \ldots u_p$, which contradicts the maximality of $P$.
            \item Otherwise, $b$ is not adjacent to $P$. Since $G$ is connected and there are only two vertices not in $V(P)$, there exists $a\in A\cap N(b)$ such that $a$ is adjacent to some vertex $u_i\in B\cap V(P)$. The path $b a u_i u_{i-1} \ldots u_1 u_{i+2} \ldots u_p$ contradicts the maximality of $P$.
        \end{itemize} 
        Both cases lead to a contradiction, hence $d(u_1)\leq k-2$.
        \cqed 
    \end{proof}
    Among the vertices of $A$ that are not cut-vertices, let $x \in A$ be one with smallest degree. By Claim~\ref{cl:max_degree}, we have $d(x) \leq k-2$.
    Consider the graph $G' =  G \setminus \{x\}$, which is connected.
    We have $|E(G')| \geq (k-1)^2 +2 -(k-2) =  k^2-3k+5$.
    Furthermore, in $G'$, no vertex of $A$ has degree $1$. Suppose for sake of contradiction that $z\in A$ has degree $1$. Then $z$ is not a cut-vertex and by minimality of $d(x)$, we derive $d(x)=1$. All the other vertices of $A$ have degree at most $k$ in $G$, since $G$ is bipartite and $|B|=k$. This sums up to $k(k-2)+2< (k-1)^2 +2 \leq |E(G)|$ edges, which is a contradiction.
    We now prove the following claim.
    \begin{claim}\label{cl:nbrcutvertex}
        No vertex of $B$ is a cut-vertex of $G'$.
    \end{claim}
    \begin{proof}
        Suppose for the sake of contradiction that $G'$ has a cut-vertex $y \in B$. Then $y$ separates $G'$ into two non-empty graphs $G_1$ and $G_2$ such that $|V(G_1)\cap A| = (k-1) - |V(G_2)\cap A| = \ell $ and $|V(G_1)\cap B| = k-1 - |V(G_2)\cap B| = p $.
        Notice that $y$ needs to be adjacent to at least one vertex in $G_1$ and in $G_2$, so $1\leq \ell \leq k-2$. Moreover, every vertex of $A$ has at least $2$ neighbors in $G'$, in particular one different from $y$. Thus, $1\leq p \leq k-2$.
        The maximal number of edges of $G_1$ is $\ell p$, and for $G_2$, it is $(k-1-\ell)(k-1-p)$. As $y$ has at most $k-1$ edges, the total number of edges of $G'$ is at most $f(\ell,p):=k-1 + \ell p + (k-1-\ell)(k-1-p)$.
        We now show that $f(\ell,p)< k^2-3k+5\leq |E(G')|$, giving a contradiction. 
        Notice that $f(\ell,p)$ is linear in $\ell$ and $p$.
        Therefore, the maximal values of $f$ in $[1,k-2]^2$ are attained at the corners, when $\ell,p\in[1,k-2]$.
        Evaluate $f$ at the corners:
        \begin{itemize}
            \item $f(1,1) = k-1 + 1 + (k-2)^2 = k^2-3k+4$,
            \item $f(k-2,k-2) = f(1,1)$ by symmetry,
            \item $f(1,k-2) = k-1+ k-2 + k-2 = 3k-5 \leq k^2-3k+4$, and
            \item $f(k-2,1)=f(1,k-2)$ by symmetry.
        \end{itemize}
        This proves that $f(\ell,p)\leq k^2-3k+4 < |E(G')|$, a contradiction.

        \vspace{-0.4cm}\cqed 
    \end{proof}

    \begin{claim}\label{cl:degvertex}
        There exist two vertices of $B$ of degree at least $k-1$ in $G$.
    \end{claim}
    \begin{proof}
        If there is only one vertex of degree at least than $k-1$, then the total number of edges of $G$ is less than $(k-1)(k-2) + k = (k-1)^2 +1 < (k-1)^2 +2 $. So, there is another vertex of degree at least $k-1$.
        \cqed 
    \end{proof}
    From Claim~\ref{cl:degvertex} and Claim~\ref{cl:nbrcutvertex}, there exist two vertices $y_1,y_2 \in B$ such that $d_G(y_1) \geq k-1$, $d_G(y_2) \geq k-1$ and that are not cut-vertices of $G'$.
    Without loss of generality, either both $y_1$ and $y_2$ are adjacent to $x$, or $y_1$ is not adjacent to $x$.
    Consider the connected balanced bipartite graph $G'' =G' - y_1$ on $2k -2$ vertices.
    Suppose for the sake of contradiction that $G''$ is not $P_{2k-2}\text{-free}$.
    Consider the (hamiltonian) path $P' = v_1\dots v_{2k-2}$ in $G''$, and assume that $v_1\in A$ without lost of generality, so that $v_{2k-2}\in B$.
    Notice that $P'$ contains $k-1$ vertices of $A$ and $k-1$ vertices of $B$. 
    There are two possibilities:
    \begin{itemize}
        \item If $y_1$ is adjacent to $x$, then by assumption on $y_1$, we have $d_G(x)\geq 2$. Then, it is adjacent to some $v_i$.
        If $i=2k-2$, then the path $yxv_{2k-2}\dots v_1$ is a hamiltonian path of $G$, contradiction.
        Therefore, $2\leq i\leq 2k-3$ and at least one of $v_{i-1}$ and $v_{i+1}$ is adjacent to $y_1$, because $d_G(y_1)\geq k-1$ ($y_1$ has at most one non neighbor in $V(P')\cap A$).
        Suppose it is $v_{i+1}$. The other case is similar.
        The path $v_1 \dots v_i x y_1 u_{i+1} \dots v_{2k-2}$ is a hamiltonian path of $G$, contradiction.
        \item Else, $y_1$ is non-adjacent to $x$ and $y_1$ is adjacent to every vertex of $A$ in $P'$.
        As $G$ is connected, $x$ is adjacent to some $v_i$. If $i=2k-2$, then $y_1 v_1 \dots v_{2k-2} x$ is a hamiltonian path of $G$, contradiction.
        Therefore, $2\leq i\leq 2k-3$ and both $v_{i-1}$ and $v_{i+1}$ are adjacent to $y_1$.
        The path $x v_i v_{i-1} \dots v_1 y_1 v_{i+1} \ldots v_{2k-2}$ is a hamiltonian path of $G$, contradiction.
    \end{itemize}
    Both cases lead to a contradiction. Hence $G''$ is $P_{2k-2}\text{-free}$.
    By induction hypothesis, we have that $|E(G'')| \leq (k-2)^2 +1$, and \[|E(G)|\leq |E(G'')| + (k-1)+(k-2) \leq (k-2)^2 + 2k-2 = (k-1)^2 + 1.\]
    This finishes the proof.
\end{proof}

\section{Brooms}\label{sec:brooms}

We start by proving the following two lemmas : 

\begin{lemma}\label{lem:all_path2}
     For every $k \geq 1$, for any connected bipartite graph $G$ of part sizes $a$ and $b$ with $b,a \geq 1$, if there exists a vertex of $G$ which is not the endpoint of a path of size $2k$, then $G$ has at most $(k-1)(b+a)$ edges.
\end{lemma}

\begin{lemma}\label{lem:lem2}
    For every $k,d \geq 1$, for any $S_{2k+1,d*1}$-free connected bipartite graph $G$ of part size $a$ and $b$ with $b\geq a \geq k$, if $G$ has at least $ k(b+a) +1 $ edges, then every vertex of $G$ has a degree at most $k+d-1$.
\end{lemma}

\begin{proof}[Proof of Lemma~\ref{lem:all_path2}]
    Let $G$ be a connected bipartite graph with parts of size $a$ and $b$ with $b \geq a$, such that there exists a vertex of $G$ which is not the endpoint of a path of size $2k$.
    We proceed by induction on $a+b = |V(G)|$.
    \begin{itemize}
        \item If $a < k$ or $b < k$ then $|E(G)| \leq (k-1)  \max\{a,b\} \leq (k-1)(a+b)$.
        \item If $b \geq a \geq k$, then $b+a \geq 2k$. Let $x$ be a vertex of $G$ which is not the endpoint of a path of size $2k$ and let $P=xu_1 \ldots u_p$ be a path of maximum size starting with $x$. Then $p+1\leq 2k-1$ and thus $d(u_p) \leq k-1$. Let $G' = G \setminus\{u_p\}$. Note that $G'$ is a bipartite graph such that there exists a vertex which is not the endpoint of a path of size $2k$. Moreover $G'$ is also connected because $u_p$ was the endpoint of a path of maximal size. By induction hypothesis, we have $|E(G')| \leq (k-1)(a+b-1)$ so $|E(G)| \leq |E(G')| + k-1 \leq (k-1)(a+b)$.
    \end{itemize}
\end{proof}

\begin{proof}[Proof of Lemma~\ref{lem:lem2}]
    Let $G$ be an $S_{2k+1,d*1}$-free connected bipartite graph of part size $a$ and $b$ with $b\geq a$ and $|E(G)| \geq k(b+a) +1$. By Lemma~\ref{lem:all_path2}, every vertex of $G$ is the endpoint of a path of size $2k+1$ (in fact $2k+2$, but this is unnecessary for the rest of the proof). 
    
    Let $x$ be a vertex of $G$ and let $P$ be a path of length $2k+1$ starting with $x$. The vertex $x$ can only have $d-1$ neighbours not on $P$, otherwise $S_{2k+1,d*1}$ would appear in $G$, and since $G$ is bipartite, the vertex $x$ can have at most $k$ neighbours on $P$. So $x$ has degree at most $k+d-1$.
    
\end{proof}

We can now prove the main theorem of this section.

\thmBrooms*
On a side note, the case where $n\leq k+d$ is easy, as the complete bipartite graph $K_{n,n}$ cannot contain $S_{2k,d*1}$.

We start by proving the lower bound:
\begin{lemma}\label{lem:lower_bound2}
     For every $k,d \geq 2$ and for every $n$, we have $ex_{b,c}(n,n,S_{2k+1,d*1}) \geq ex_{b,c}(n,n,S_{2k,d*1}) \geq nd $.
\end{lemma}

\begin{proof}[Proof of Lemma~\ref{lem:lower_bound2}]
    The first part of the inequality is immediate, we focus on the second part.
    We build a connected balanced bipartite graph $G$ of size $2n$ as follows. 
    Let $A = \{u_0,\ldots,u_{n-1}\}$ and $B = \{v_0,\ldots,v_{n-1}\}$ be the two sides of the bipartition of $G$. We start from the empty graph and for every $1\leq i\leq n$ and every $0\leq j\leq d-1$, we add $u_iv_{(i+j \mod{n})}$ to $E(G)$ .
    We have $|E(G)| = nd$. Moreover, $G$ has maximum degree $d$, so $S_{2k,d*1}$ cannot be a subgraph of $G$.  
\end{proof}

Now, we focus on the upper bound:
\begin{lemma}\label{lem:upper_bound2}
    For every $k,d \geq 2$ and for every $n$ such that $n\geq d^2/2$ and $d>2k$, we have $ex_{b,c}(n,n,S_{2k+1,d*1}) \leq nd $.
\end{lemma}

\begin{proof}[Proof of Lemma~\ref{lem:upper_bound2}]
    Let $G$ be a $S_{2k+1,d*1}$-free connected balanced bipartite graph of size $2n$ and suppose that $|E(G)| > nd$.  We have $|E(G)| > nd \geq  2kn$, so by Lemma~\ref{lem:lem2} we know that $G$ has maximum degree at most $k+d-1$.
    Let $x$ be a vertex in $G$ of maximum degree. As $|E(G)| > nd$, we must have $d(x)\geq d+1$. Let $u_1, \ldots ,u_{d+1}$ be $d+1$ arbitrary neighbours of $x$ and let $G' = G \setminus\{x, u_1,\ldots, u_{d+1}\}$. 
    Let $C_1, \ldots, C_p$ be the connected components of $G'$ (possibly $p=1$). For each $i\leq p$, since $G$ is connected, $C_i$ is adjacent in $G$ to at least one vertex of the set $\{x, u_1,\ldots, u_{d+1}\}$. Consequently, in each $G[C_i]$, not every vertex is the endpoint of a path of length $2k$: otherwise we would have $S_{2k+1,d*1}$ as a subgraph. 
    By Lemma~\ref{lem:all_path2}, we derive that for every $i\leq p$, we have $|E(C_i)| \leq (k-1)|V(C_i)|$. Thus we obtain: $$ |E(G')| = \sum_{i=1}^{p} |E(C_i)| \leq (k-1)\sum_{i=1}^{p}|V(C_i)| = (k-1) |V(G')| = (k-1)(2n - (d+2)).$$

    Furthermore, the graph $G'$ was obtained from $G$ by removing $d+2$ vertices, each of degree at most $k+d-1$ in $G$. By removing the double-counting of the edges of the form $xu_i$, we obtain $|E(G)| \leq |E(G')| + (d+2)(d+k-1) - (d+1) \leq (k-1)(2n - (d+2)) + (d+2)(d+k-2) + 1 = 2kn - 2n +d^2 +d - 1$.
    As $n \geq d^2/2$ and $d> 2k$, we have $|E(G)|\leq 2kn +d -1 < (2k+1)n \leq dn \leq |E(G)|$, a contradiction.
\end{proof}

In the cases where either we do not have $n\geq d^2/2$ and $d>2k$, but have $d<k$, the following bound holds.
\thmBroomsTwo*

\begin{proof}[Proof of Theorem~\ref{th:d<n}]
    Let $G$ be a $S_{2k+1,d*1}$-free connected balanced bipartite graph on $2n$ vertices and suppose that $|E(G)| \geq 2nk + 1$.
    According to Lemma~\ref{lem:lem2}, all the vertices of $G$ have degree at most $d+k-1$, so $|E(G)| \leq n(k+d-1) \leq 2n(k-1)$ which is a contradiction.
\end{proof}
Slightly worse bounds can be obtained for $d\leq 2k$.

\bibliographystyle{alpha}
\bibliography{coolnew}

\end{document}